\documentclass{amsart}

\usepackage{amsmath,amsthm,amsfonts,amssymb,amscd,verbatim,delarray,url}
\usepackage[arrow,matrix]{xy}

\DeclareMathOperator{\PSL}{PSL} \DeclareMathOperator{\SL}{SL}
 
\DeclareMathOperator{\tr}{tr} 
 \DeclareMathOperator{\Spec}{Spec}
\DeclareMathOperator{\id}{id}

\begin{document}

\theoremstyle{plain}
\newtheorem{theorem}{Theorem}[section]
\newtheorem{conj}[theorem]{Conjecture}
\newtheorem{corollary}[theorem]{Corollary}
\newtheorem{prop}[theorem]{Proposition}
\newtheorem{Lemma}[theorem]{Lemma}

\theoremstyle{definition}
\newtheorem{definition}[theorem]{Definition}
\newtheorem{quest}[theorem]{Question}
\newtheorem{remark}[theorem]{Remark}
\newtheorem{example}[theorem]{Example}
\newtheorem*{answer}{Answer}
\newtheorem{notation}[theorem]{Notation}

\newcommand{\thmref}[1]{Theorem~\ref{#1}}
\newcommand{\secref}[1]{Section~\ref{#1}}
\newcommand{\subsecref}[1]{Subsection~\ref{#1}}
\newcommand{\lemref}[1]{Lemma~\ref{#1}}
\newcommand{\corolref}[1]{Corollary~\ref{#1}}
\newcommand{\exampref}[1]{Example~\ref{#1}}
\newcommand{\remarkref}[1]{Remark~\ref{#1}}
\newcommand{\defnref}[1]{Definition~\ref{#1}}
\newcommand{\propref}[1]{Proposition~\ref{#1}}

\newcommand{\BZ}{\mathbb{Z}}
\newcommand{\BC}{\mathbb{C}}
\newcommand{\BN}{\mathbb{N}}
\newcommand{\BP}{\mathbb{P}}
\newcommand{\BF}{\mathbb{F}}
\newcommand{\BA}{\mathbb{A}}
\newcommand{\BQ}{\mathbb{Q}}
\newcommand{\BM}{\mathbb{M}}
\newcommand{\BX}{\mathbb{X}}
\newcommand{\BY}{\mathbb{Y}}

\newcommand{\ep}{\epsilon}
\newcommand{\al}{\alpha}
\newcommand{\be}{\beta}
\newcommand{\ga}{\gamma}
\newcommand{\de}{\delta}
\newcommand{\la}{\lambda}
\newcommand{\om}{\omega}
\newcommand{\vp}{\varphi}
\newcommand{\st}{\sigma}

\newcommand{\vareps}{\varepsilon}
\newcommand{\G}{\Gamma}

\newcommand{\ov}{\overline}

\newcommand{\fc}{\frac}

\newcommand\cR{\mathcal R}
\newcommand\cP{\mathcal P}

%\begin{document}

%\begin{frontmatter}

\title[Criteria for equidistribution for word equations on
$\SL(2)$]
{Criteria for equidistribution of solutions of word equations on
$\SL(2)$}

\author{Tatiana Bandman}
\address{Department of Mathematics, Bar-Ilan University, 5290002 Ramat Gan,
Israel} \email{bandman@macs.biu.ac.il}

\author{Boris Kunyavski\u\i }
\email{kunyav@macs.biu.ac.il}

\begin{abstract}
We study equidistribution of solutions of word equations of the form
$w(x,y)=g$ in the family of finite groups $\SL(2,q)$. We provide
criteria for equidistribution in terms of the trace polynomial of
$w$. This allows us to get an explicit description of certain
classes of words possessing the equidistribution property and show
that this property is generic within these classes.
\end{abstract}

%\begin{keyword} word map \sep finite group of Lie type \sep equidistribution \sep
%trace polynomial \MSC 11E57 \sep 12E05 \sep 20G40
%\end{keyword}

%\end{frontmatter}

\maketitle

\section {Introduction}\label{sec:intro}

Equidistribution of solutions of various (systems of) diophantine
equations has been remaining one of central topics in number theory,
arithmetic geometry, ergodic theory. It is not our goal to review
vast literature in the area. The reader interested in evolution of
ideas in this fascinating domain of mathematics may find instructive
to overview materials of ICM's, starting from the foundational
address by Linnik (Stockholm, 1962) until impressive contributions
of the past two decades: Margulis, Sarnak (Kyoto, 1990); Dani,
Ratner (Z\"urich, 1994); Eskin (Berlin, 1998); Ullmo (Beijing,
2002); Einsiedler--Lindenstrauss, Michel--Venkatesh, Tschinkel
(Madrid, 2006); Oh, Shah (Hyderabad, 2010). Each of the approaches
mentioned above assumes its own understanding of the notion of
equidistribution. What most of them share in common is focusing on
certain group actions arising in a natural way and allowing one to
combine methods of number theory and dynamical systems with
group-theoretic considerations.

Let us describe the circle of problems we are interested in. First,
we want to study {\it polynomial matrix equations}. In the most
general form, one can consider equations of the form $P(A_1,\dots
,A_m, X_1,\dots , X_d)=0$ where $n\times n$-matrices $A_1,\dots,
A_m$ with entries from a ring $R$ are {\it given}, $X_1,\dots ,X_d$
are {\it unknowns}, and $P$ is an {\it associative noncommutative}
polynomial. We, however, restrict our attention to a particular
class of equations of the form $P(X_1,\dots , X_d)=A$ where $A$ is a
given matrix, $X_1,\dots ,X_d$ are unknowns, and a solution must
belong to a fixed subset $\mathcal M\subset {\text{\rm M}}(n,R)^d$.
There are several cases where such an equation has a solution for a
``generic'' $A$ (here $R=K$ is an algebraically closed field):
\begin{itemize}
\item
$\mathcal M=G(K)^d$ where $G(K)$ is the group of rational points of
a connected semisimple algebraic group and $P=w\ne 1$ is a
nontrivial word (=monomial in $X_1, X_1^{-1},\dots ,X_d, X_d^{-1}$)
(Borel \cite{Bo}, Larsen \cite{La});

\item
$\mathcal M=\mathfrak g^d$ where the Lie algebra $\mathfrak g$ of a
semisimple algebraic $K$-group and a Lie polynomial $P$ satisfy some
additional assumptions (Bandman, Gordeev, Kunyavski\u\i \ and
Plotkin \cite{BGKP});

\item
$\mathcal M={\text{\rm M}}(n,R)^d$ and $P$ satisfies some additional
assumptions (Kanel-Belov, Malev and Rowen \cite{KBMR}).
\end{itemize}

If $R=\mathbb Z$, in all these cases we may interpret the situation
as follows: the generic fibre of the morphism $\mathbb P\colon
\mathbb M^d\to\mathbb M$ of $\mathbb Z$-schemes, induced by the
polynomial $P$, is a {\it dominant} morphism of $\mathbb Q$-schemes.

One can ask whether the situation is similar in {\it special} fibres
of the morphism $P$. As the notion of dominance does not make much
sense for finite sets, we would like to formalize the following
phenomena:

\begin{itemize}
\item
the maps $P_q\colon (M_q)^d\to M_q$ %of the morphism $\mathbb P$
have ``asymptotically large'' images;
\item
the number $\#\{(A_1,\dots ,A_d)\in (M_q)^d: P_q(A_1,\dots
,A_d)=A\}$ (where $q=p^n$; $p=2,3,5,\dots$; $A$ runs over a
``large'' subset of $M_q$) is, in some reasonable sense, almost
independent of $A$.
\end{itemize}

(Here $M_q$ denotes the set of $\mathbb F_q$-points of the fibre of
the scheme $\mathbb M$ at $q$, and $P_q$ is the fibre of the
morphism $\mathbb P$ at $q$.)

The conditions formulated above mean that the equations $P(X_1,\dots
,X_d)=A$, with the right-hand side running, for each $q$, over
``almost whole'' set $M_q$, have many and almost equally many
solutions in $(M_q)^d$, respectively. We shall call such morphisms
{\it $p$-almost equidistributed}, or {\it almost equidistributed}
(depending on whether $p$ in the second condition is or is not
fixed); the word ``almost'' will often be dropped. See Section
\ref{results} for precise definitions.

According to Larsen \cite{La}, Larsen and Shalev \cite{LS1}, for any
word $w\ne 1$ and any family of Chevalley groups $G_q$ of fixed
type, the images of the maps $P_{w,q}\colon (G_q)^d\to G_q$ are
``asymptotically large''. Note, however, that for any individual
$G_q$ the image of $P_{w,q}$ may be very small: say, $w$ may be
identically 1 on $(G_q)^d$; moreover, even if this is not the case,
then, according to an observation of Kassabov and Nikolov \cite{KN}
(see also a subsequent paper of Levy \cite{Le1}), the image of $w$
may consist only of a single conjugacy class together with the
identity element. Recently Lubotzky \cite{Lu} proved that such a
phenomenon can happen in any finite simple group, for any conjugacy
class; Levy \cite{Le2} extended this result to some almost simple
and quasisimple groups.

Our main result (\thmref{equi}) provides a necessary and sufficient
condition on the word $w$ in two variables under which the 
morphism ${\mathbb P}_w\colon \SL_2\times \SL_2\to
\SL_2$ is almost equidistributed. This result can be viewed, on the
one hand, as a refinement (in the $\SL_2$-case) of equidistribution
theorems of Larsen and Pink \cite{LP}, Larsen and Shalev \cite{LS2},
Larsen, Shalev and Tiep \cite{LST} on general words $w$ and general
Chevalley groups $G$, and, on the other hand, as a generalization of
equidistribution theorems for some particular words: Garion and
Shalev \cite{GS} (commutator words on any $G$), Bandman, Garion and
Grunewald \cite{BGG} (Engel words on $\SL_2$), Bandman and Garion
\cite{BG} (positive words on $\SL_2$). As a consequence, we obtain a
somewhat surprising conclusion: if the word morphism as above has a
large image (in the sense that for almost all $q$ the image of
$P_{w,q}$ contains all noncentral semisimple elements of
$\SL(2,q)$), then it is almost equidistributed (in the terminology
of the preceding paragraph, ``many'' implies ``almost equally
many'').

Acting in the spirit of \cite{GS}, we deduce a criterion for
$w\colon \SL_2\times \SL_2\to \SL_2$ to be {\it almost
measure-preserving}.

Note that certain word maps are measure-preserving in a much
stron\-ger sense. Namely, if $w$ is {\it primitive}, i.e., is a part
of a basis of the free $d$-generated group $F_d$, then the
corresponding word map $G^d\to G$ is measure-preserving for {\it
every finite group} $G$, i.e., all fibres of this map have the same
cardinality. Only primitive words possess this property, this was
proven for $d=2$ by Puder \cite{Pu} and extended to arbitrary $d$ by
Puder and Parzanchevski \cite{PP}. (Note that the word map $P_{w}$
induced by a primitive word $w$ is obviously surjective.) It is well
known (see, e.g., Myasnikov and Shpilrain \cite{MS}) that primitive
words are asymptotically rare (negligible, in the terminology of
Kapovich and Schupp \cite{KS}). We are looking for criteria for
equidistribution for more general words.

The criteria we are talking about are formulated in terms of the
{\it trace polynomial} of the word $w$. It turns out (see our main
results in Section \ref{results}; they are proved in Section
\ref{proofs}) that ``good'' (equidistributed, measure-preserving)
words are essentially those whose trace polynomial cannot be
represented as a composition of two other polynomials. Since a
``bad'' trace polynomial tends to be the trace polynomial of some
power word (see Section \ref{sec:comp}), we conclude (see Section
\ref{sec:gen}) that within certain natural classes of words a
``random'' word is ``good'' (``good'' words, i.e., those whose trace
map is $p$-equidistributed for all but finitely many primes $p$,
form an exponentially generic set, in the sense of \cite{KS}).

\section{Main results}\label{results}

We start with precise definitions of notions described in the
introduction. We will follow the approach to equidistribution
adopted in \cite{GS}:

\begin{definition} (cf. \cite[\S 3]{GS}) \label{def:equifin}
Let $f\colon X\to Y$ be a map between finite non-empty sets, and let
$\vareps >0$. We say that $f$ is {\em $\vareps$-equidistributed} if
there exists $Y'\subseteq Y$ such that
\begin{enumerate}
\item[(i)] $\#Y' > \#Y(1-\vareps )$;
\item[(ii)] $|f^{-1}(y)-\frac{\#X}{\#Y}| < \vareps \frac{\# X}{\# Y}$ for all $y\in Y'$.
\end{enumerate}
\end{definition}

Our setting is as follows. Let a family of maps of finite sets
$P_q\colon X_q\to Y_q$ be given for every $q=p^n$. Assume that for
all sufficiently large $q$ the set $Y_q$ is non-empty. For each such
$q$ take $y\in Y_q$ and denote
$$
P_y=\{x\in X_q : P_q(x)=y\}.
$$

\begin{definition} \label{def:equimor-p}
Fix a prime $p$. With the notation as above, we say that the family
$P_q\colon X_q\to Y_q$, $q=p^n$, is {\em $p$-equidistributed} if
there exist a positive integer $n_0$ and a function
$\vareps_p\colon\BN\to\BN$ tending to $0$ as $n\to\infty$ such that
for all $q=p^n$ with $n>n_0$ the set $Y_q$ contains a subset $S_q$
with the following properties:
\begin{enumerate}
\item[(i)] $\#S_q < \vareps_p(q) \, (\#Y_q)$;
\item[(ii)] $|\#P_y-\frac{\#X_q}{\#Y_q}| < \vareps_p(q) \frac{\#X_q}{\#Y_q}$ for all $y\in Y_q\setminus S_q$.
\end{enumerate}
\end{definition}

\begin{remark}
Definition \ref{def:equimor-p} means that for $q=p^n$ large enough,
the map $X_q\to Y_q$ is $\vareps_p(q)$-equidistributed, in the sense
of Definition \ref{def:equifin}.
\end{remark}

\begin{definition} \label{def:equimor}
We say that the family $P_q\colon X_q\to Y_q$ is {\em
equidistributed} if it is $p$-equidistributed for all $p$ and there
exists a function $\vareps\colon\BN\to\BN$ tending to $0$ as
$n\to\infty$ such that for every $p$ and every $q=p^n$ large enough,
we have $\vareps_p(q)\le\vareps (q)$.
\end{definition}

Let us now consider the case where $Y_q=G_q$ is a Chevalley group
over $\BF_q$, $X_q=(G_q)^d$ is a direct product of its $d$ copies
($d\ge 2$ is fixed), and $P_q=P_{w,q}\colon (G_q)^d\to G_q$ is the
map induced by some fixed word $w\in F_d$: to each $d$-tuple
$(g_1,\dots ,g_d)$ we associate the value $w(g_1,\dots ,g_d)$.

In the present paper we focus our attention on a particular case
$d=2$, $G_q= \SL(2,q)$. It is convenient to view the maps
$P_{w,q}\colon \SL (2,q)\times \SL (2,q)\to \SL (2,q)$ as fibres of
the morphism  $\BP_w\colon \SL_{2,\BZ}\times \SL_{2,\BZ}\to
\SL_{2,\BZ}$ of group schemes over $\BZ$. We say that the morphism
$\BP_w$ (or, for brevity, the word $w$) is equidistributed (or
$p$-equidistributed) if so is the family $P_{w,q}$.

In such a situation, there is a natural way to associate to any word
$w=w(x,y)\in F_2$ its {\it trace polynomial}. This construction goes
back to the 19th century (Vogt, Fricke, Klein), see, e.g., \cite{Ho}
for a modern exposition. For $G=\SL(2,k)$ ($k$ is any commutative
ring with 1) denote by $\tr(w)\colon G^2\to G$ the trace character,
$(g_1,g_2)\mapsto \tr (w(g_1,g_2))$. Then $\tr (w)=f_w(s,u,t)$ where
$f_w \in \BZ[s,u,t]$ is an integer polynomial in three variables
$s=\tr (x)$, $u=\tr (xy)$, $t= \tr(y)$. We denote by the same
letters the induced morphisms of affine $\BZ$-schemes

$$f_w\colon \BA^3_{s,u,t}=\Spec \BZ[s,u,t]\to \BA^1_{z}=\Spec \BZ[z],$$
of affine $\ov{\BF}_p$-schemes:
$$f_{w,p}\colon \Spec \ov{\BF}_p[s,u,t]\to \Spec \ov{\BF}_p[z], $$
and also maps of sets of $\ov{\BF}_p$-points:
$$f_{w,p}\colon \BA^3_{s,u,t}(\ov{\BF}_p)\to \BA^1_{z}(\ov{\BF}_p)$$
(here  $\BA^N_{x_1,\dots, x_N}$ stands for affine space with
coordinates $x_1,\dots, x_N$).

Our criteria for equidistribution of $w$ will be formulated in terms
of the polynomial $f_w$. Some recollections and definitions on
polynomials are on order.

\begin{definition} \label{perm-pol}
Let $\BF$ be a finite field. We say that $h\in \BF[x]$ is a
permutation polynomial if the set of its values $\{h(z)\}_{z\in\BF}$
coincides with $\BF$.
\end{definition}

\begin{theorem} \cite[Theorem~7.14]{LN} \label{th:perm}
Let $q=p^n$. A polynomial $h\in \BF_q[x]$ is a permutation
polynomial of all finite extensions of $\BF_q$ if and only if
$h=ax^{p^k}+b,$ where $ a\ne 0$ and $k$ is a non-negative integer.
\end{theorem}

The following notions are essential for our criteria.

\begin{definition}\label{def3}
Let $\BF$ be a field. We say that a polynomial $P\in \BF[x_1,\dots,
x_n]$ is  {\it $\BF$-composite} if there exist  $Q\in \BF[x_1,\dots,
x_n]$, $\deg Q\geq 1$, and $h\in \BF[z]$, $\deg h\geq 2$, such that
$P=h\circ Q.$ Otherwise, we say that $P$ is {\it
$\BF$-noncomposite}.
\end{definition}

Note that if $\mathbb E/\BF$ is a separable field extension, it is
known \cite[Theorem 1 and Proposition 1]{AP} that $P$ is
$\BF$-composite if and only if $P$ is $\mathbb E$-composite. In
particular, working over perfect ground fields, we may always
assume, if needed, that $\BF$ is algebraically closed.

\begin{definition}\label{def4}
Let $P\in \BZ[x_1,\dots, x_n]$. Fix a prime $p$.
\begin{itemize}
\item We say that $P$ is {\it $p$-composite} if
the reduced polynomial $P_p\in \BF_p[x_1,\dots, x_n]$ is
$\BF_p$-composite. Otherwise, we say that $P$ is {\it
$p$-noncomposite}.
\item We say that a $p$-composite polynomial $P$ is {\it $p$-special} if,
in the notation of Definition $\ref{def3}$, $P_p=h\circ Q$ where
$h\in \BF_p[x]$ is a permutation polynomial of all finite extensions
of $\BF_p.$
\end{itemize}
\end{definition}

\begin{definition}\label{def5}
We say that a polynomial  $P\in \BZ[x_1,\dots, x_n]$ is {\it almost
noncomposite} if for every prime $p$ it is either $p$-noncomposite
or $p$-special. Otherwise we say that $P$ is {\it very composite}.
\end{definition}

\begin{remark} \label{r1} If a polynomial  $P\in \BZ[x_1,\dots, x_n]$
is $\BQ$-noncomposite, it is $p$-noncomposite for all but finitely
many primes $p$ \cite[2.2.1]{BDN}. If $P\in \BZ[x_1,\dots, x_n]$ is
$\BQ$-composite, it is very composite.
\end{remark}

\begin{example}\label{ex:st-u}
Consider the family of Dickson polynomials ${\mathcal D}_n(x,a)$.
Denote $D_n(x)={\mathcal D}_n(x,1)$. We have $D_n(x)=2T_n(x/2)$
where $T_n(x)$ is the $n^{th}$ Chebyshev polynomial. If $n$ is not
prime then $D_n$ is very composite (see, e.g., Section
\ref{sec:comp} below). If $n=p$ is prime, then $D_n$ is almost
noncomposite and $p$-special since $D_p(x)=x^p$ in $\BF_p[x]$.
\end{example}

We can now formulate our main results.

\begin{theorem}\label{p-equi}
Let $w\in F_2$. The morphism $\BP_w\colon\SL_{2,\BZ}\times \SL_{2,\BZ}\to
\SL_{2,\BZ}$ is $p$-equidistributed if and only if the trace
polynomial $f_w$ is either $p$-noncomposite  or $p$-special.
\end{theorem}

\begin{theorem}\label{equi}
Let $w\in F_2$. The morphism $\BP_w\colon\SL_{2,\BZ}\times \SL_{2,\BZ}\to
\SL_{2,\BZ}$ is equidistributed if and only if the trace polynomial
$f_w$ is almost noncomposite.
\end{theorem}

\begin{corollary} \label{cor-main}
Let $w(x,y)=x^{a_1}y^{b_1}\dots x^{a_r}y^{b_r}$ be a reduced word
such that we have $f_w(s,u,t)=D_r(q(s,u,t))$ over $\BQ$. Then
$w(x,y)=(x^{a_1}y^{b_1})^r$.
\end{corollary}

For a given word $w\in F_2$, let us now consider the family of
groups $\hat G_q=\PSL(2,q)$ and the corresponding word maps $\hat
P_{w,q}\colon \hat G_q\times \hat G_q\to \hat G_q$.

\begin{prop} \label{SL-PSL}
If the morphism $\BP_w\colon\SL_{2,\BZ}\times \SL_{2,\BZ}\to
\SL_{2,\BZ}$ is equidistributed $($or $p$-equidistributed$)$, then
so is the family $\hat P_{w,q}\colon \hat G_q\times \hat G_q\to \hat
G_q$.
\end{prop}

\section{Proofs} \label{proofs}

Fix a word $w$ in $F_2.$  We slightly change the general notation,
and for a group $\G$ and $g\in\G$ we denote
$$W_{g,\G}=\{(x,y)\in \G\times\G :w(x,y)=g\}.$$
We will omit the subscript $\G$ when no confusion may arise. For
$\G=G_q=\SL (2,q)$ we denote this set by $W_{g,q}$ (or just $W_g$).

Since $\#G_q=q(q^2-1)$, we will replace, if needed, $\#G_q$ by $q^3$
in all asymptotic estimates.

\begin{proof}[Proof of Theorem $\ref{p-equi}$]

Slightly rephrasing Definition \ref{def:equimor-p}, we are going to
prove that there exist positive numbers $n_0$, $A$, $B$, $\alpha$,
$\beta$, all independent of $g\in G_q$, such that for every
$q>q_0=p^{n_0}$ there exists $S_q\subset G_q$ with the following
properties:
\begin{equation} \label{Gq:equi}
\begin{aligned}
{\text{\rm{(i) }}} &\#S_q/q^3< Aq^{-\alpha};\\
{\text{\rm{(ii) }}} &{\text{\rm{ for every }}}g\in T_q:=G_q\setminus
S_q {\text{\rm{ we have }}}
\left|\frac{\#W_{g,q}}{q^3}-1\right|<Bq^{-\beta}.
\end{aligned}
\end{equation}
Indeed, this is enough for proving that $w$ is $p$-equidistributed:
in Definition \ref{def:equimor-p} one can then take
$\vareps_p(q):=\max\{Aq^{-\alpha}, Bq^{-\beta}\}.$

Towards this end, we will use the following commutative diagram:
\begin{equation} \label{diagram}
\begin{CD}
G_q\times G_q @>P_{w,q}>> G_q \\
@V\pi VV            @VV\tr V \\
\BA_{s,u,t}^3(\BF_q)       @>f_{w,q}>>  \BA_z^1(\BF_q)
\end{CD}
\end{equation}
where
\begin{equation} \label{pi}
\pi (x,y)=(\tr (x), \tr(xy), \tr (y)).
\end{equation}

``Typical'' fibres of the maps in this diagram should consist of
$O(q^3)$ elements (for $P_{w,q}$ and $\pi$), and of $O(q^2)$
elements (for $\tr$ and $f_{w,q}$). Below we will show how to attain
this with error term of order $O(q^{-\beta})$ by throwing away
$O(q^{-\alpha})$ elements.

\medskip

We will use an explicit Lang--Weil estimate of the following form:
if $H\subset \BA^3_{\BF_q}$ is
an {\it absolutely irreducible} hypersurface of degree $d$, then
$$
|\#H(\BF_q)-q^2|\leq (d-1)(d-2)q^{3/2}+12(d+3)^4q
$$
(see, e.g., \cite[Remark~11.3]{GL}), or, equivalently,
$\#H(\BF_q)=q^2(1+r_1)$ with
\begin{equation}\label{est0}
|r_1|\leq q^{-1/2}[(d-1)(d-2)+12(d+3)^4q^{-1/2}].
\end{equation}

(The remainder term $r_1=r_1(H)$, as well as all remainder terms in
the sequel, depend on the hypersurface under consideration. To ease
the notation, we do not include this dependence in formulas.)

For $d>4 $ and $q>16,$  equation \eqref{est0} gives
\begin{equation}\label{est01}
|r_1| < q^{-1/2}(d^2+12\cdot2^4d^4/4) < d^4 q^{-1/2}(1/d^2+48) < 50d^4 q^{-1/2}.
\end{equation}
Moreover, if $d>4 $ and $q > 4(50d^4)^2,$ then $|r_1| < 1/2.$
This remains true also for $d\le 3$. Without loss of generality, we
may and will assume that the latter inequality is valid.

\medskip

{\bf Step 1.  Suppose that the polynomial $f_w$ is $p$-noncomposite.}

\medskip

Denote the degree of $f_w$ by $d$, the degree of the reduced
polynomial $f_{w,p}$ is then at most $d$. Consider the corresponding
reduced map $f_{w,p}\colon \BA^3_{s,u,t}(\ov{\BF}_p)\to
\BA^1_{z}(\ov{\BF}_p)$.

Denote by $\sigma (f_{w,p})$ the spectrum of $f_{w,p}$, i.e., the
set of all points $z\in \BA^1_{z}(\ov{\BF}_p)$ such that the
hypersurface $H_z\subset\BA^3_{s,u,t}(\ov{\BF}_p)$, defined by the
equation $f_w(s,u,t)=z$, is reducible. By a generalized
Stein--Lorenzini inequality \cite{Na}, this set contains at most
$d-1$ points. The same is true for each $\sigma_q(f_w):=\sigma
(f_{w,p})\cap \BF_q$. Without loss of generality, we may and will
assume that $\pm 2$ are inside $\sigma_q(f_w)$ (by enlarging
$\#\sigma_q(f_w)$ to $d+1$).

Let $z\in\BA^1_{z}(\ov{\BF}_p)\setminus  \sigma(f_{w,p}).$ Then
$H_z$ is an irreducible hypersurface and hence \eqref{est0}, \eqref{est01}
are valid for $H_z.$

\begin{Lemma}\label{BG}
Let $H \subset\BA^3_{s,u,t}(\ov{\BF}_p)$ be a hypersurface of degree
$d$. Let $D(s,u,t)=(t^2-4)(s^2-4)(s^2+t^2+u^2-ust-4)$, and let
$\Delta\subset\BA^3_{s,u,t}$ be defined by the equation $D=0$.
Assume that $H\not\subset  \Delta.$ Then $($see $\eqref{pi})$ we
have $\#\pi^{-1}(H)(\BF_q)=\#H(\BF_q) q^3(1+r_2),$ where $|r_2| <
157 d/q.$
\end{Lemma}
\begin{proof}

We use the following fact (see \cite[Proposition 7.2]{BG}):
$$
\#\pi^{-1}(s,u,t)(\BF_q)=q^3(1+\delta_1(s,u,t)),  \ |\delta_1|\leq
3/q,
$$
if $(s,u,t)\not \in\Delta (\BF_q)$, and
$$
\#\pi^{-1}(s,u,t)(\BF_q)\leq 2q^3(1+1/q)
$$
if $(s,u,t) \in\Delta (\BF_q). $

Denote $H\cap \Delta$ by $H_{\Delta}.$ By Bezout's theorem, this is
a curve of degree at most $7d, $ hence $\#H_{\Delta}(\BF_q)\leq
7d(q+1).$  We have
$$
\begin{aligned}
\#\pi^{-1}(H)(\BF_q) & =\#\pi^{-1}(H \setminus
H_{\Delta})(\BF_q)+\#\pi^{-1}(H_{\Delta})(\BF_q) \\
& \le \#(H \setminus H_{\Delta})(\BF_q) q^3(1+\al_1)+
\#H_{\Delta}(\BF_q) q^3\al_2,
\end{aligned}
$$
where $|\al_1|\leq 3/q$ and $|\al_2|\leq 2(1+1/q)\leq 3.$ Thus
$$
\frac{\#\pi^{-1}(H)(\BF_q)}{\#H(\BF_q)} =
q^3\left[\left(1-\frac{\#H_{\Delta}(\BF_q)}{\#H(\BF_q)}\right)(1+\al_1)+\frac{\#H_{\Delta}(\BF_q)}{\#H(\BF_q)}\al_2\right]
= q^3(1+r_2)
$$ with
$$
\begin{aligned}
|r_2| & \leq \frac{\#H_{\Delta}(\BF_q)}{\#H(\BF_q)}(1+|\al_1| +
|\al_2|)+ |\al_1|
\leq \frac{7d(q+1)}{q^2(1+r_1)} (1+|\al_1|+ |\al_2|)+ |\al_1| \\
& \leq \frac{7d\cdot 2q\cdot (11/2)}{q^2/2} + \frac{3}{q}\leq
\frac{157d}{q}.
\end{aligned}
$$
\end{proof}

Let $S'_q$ be the set of all $z\in\BF_q$ such that $ H_z\subset
\Delta$ (see \lemref{BG}). This set is finite, and $\#S_q'\leq 7$
since $\Delta $ is of degree $7$ and thus cannot contain more than
$7$ irreducible components.

Let $\tau\colon  G_q\to\BA^1$ be the trace map,  $\tau(g)=\tr(g).$
We have $\#\tau^{-1}(z)\le q(q+1).$

We define $\tilde {S_q}:= \sigma_q(f_w)\cup S'_q$ and
$S_q:=\tau^{-1}(\tilde {S_q}).$ By construction,
$$
\#S_q\leq (d+8)q(q+1)\leq q^3\frac{2(d+8)}{q}.
$$
According to \lemref{BG}, for any $z\in T_q$ we have
$$\#\pi^{-1}(H_z)(\BF_q)=\#H_z(\BF_q)q^3(1+r_2)=q^5(1+r_1)(1+r_2).$$
On the other hand, all $g\in G_q$ with $\tr(g)=z\in T_q$ are
conjugate, and there are $\#\tau^{-1}(z)=q(q\pm 1)$ such elements.
Hence for every such $g$ (see diagram (\ref{diagram})), we have
$$
\#W_g=\frac{\#\pi^{-1}(H_z)(\BF_q)}{q(q\pm 1)}=\frac{q^5(1+r_1)(1+r_2)}{q(q\pm 1)}=q^3(1+r_3)
$$
with
$$|r_3|\leq
2(|r_1|+|r_2|+|r_1r_2|)\leq
2|r_1|+3|r_2|.$$

Recall that  $q\geq  4(50d^4)^2,$ hence
$$|r_3|\leq 2\cdot 50d^4q^{-1/2}+3\cdot 157d/q \leq q^{-1/2}(100d^4+1).$$

So for $q>q_0=4(50d^4)^2$, in equation (\ref{Gq:equi}) we can take
\begin{equation}\label{parameters}
A=2(d+8), \al =1, B=100d^4+1, \beta =1/2.
\end{equation}

Thus $f_w$ is $p$-equidistributed.

\begin{remark} \label{uniform}
Note that $q_0$ and all numbers in \eqref{parameters} depend only on
$w$ (through $d$, the degree of the trace polynomial $f_w$) and not
on $p$.
\end{remark}

{\bf Step 2. Suppose that the polynomial $f_w$ is $p$-composite.}

This means that $f_w(s,u,t)=h(Q(s,u,t))$ where $h\in \BF_p[x]$ is a
polynomial in one variable of degree $d_1\geq 2$ and $Q\in
\BF_p[s,u,t]$ is a noncomposite polynomial in three variables.

Consider three separate cases.

\medskip

{\bf Case 1.} $f_w$ is $p$-special, i.e., $h$ is a permutation
polynomial of all fields $\BF_q$, $q=p^n.$ For any $z\in\BF_q$ there
is a unique $x\in\BF_q$ such that the hypersurface
$H_z\subset\BA^3$, defined by the equation $f_w(s,u,t)=z$, coincides
with  the hypersurface $\tilde H_x$, defined by the equation
$Q(s,u,t)=x$. Since $Q$ is noncomposite, {\bf Step 1} implies that
$w$ is $p$-equidistributed in this case.

\begin{remark}\label{r22}
In this case, the parameters $q_0$, $A$, $B$, $\al$, $\beta$ also do
not depend on $p$. They depend on the word $w,$ this time through
the degree of $Q$ which is less than the degree of the trace
polynomial of $w.$
\end{remark}

{\bf Case 2.}   $h$ is not a permutation polynomial for  $\BF_q$,
$q=p^n.$ Then it is not a permutation polynomial for any extension
$\BF_{q^m}$ of $\BF_q.$

According to \cite{Wa}, \cite{WSC}, there exists a subset
$U_m\subset\BA^1_{z}(\BF_{q^m})$  such that
\begin{itemize}
\item $ \#U_m\geq (q^m-1)/d_1;$
\item   $h^{-1}(s)(\BF_{q^m})=\emptyset$ for every $s\in  U_m.$
\end{itemize}

It follows that $f_w^{-1}(s)(\BF_{q^m})=\emptyset$ for every $m$ and
every $s\in U_m$.  So the polynomial $\pi\circ f_{w,q^m}$ also omits
at least $(q^m-1)/d_1$ values, and hence so does $P_{w,q^m}\circ
\tr$ (see diagram (\ref{diagram})), i.e., $P_{w,q^m}(G_{q^m}\times
G_{q^m})$ contains no elements $g\in G_{q^m}$ with $\tr(g)\in  U_m.$
For every $s\in\BF_{q^m}$, $s\ne \pm 2$, the group $G_{q^m}$
contains at least $(q^m)^2-q^m$ elements with trace $s.$ Thus
$P_{w,q^m}$ omits at least
$$q^m(q^m-1)[(q^m-1)/d_1-2]\approx (q^m)^3/d_1$$ values. Hence $w$ is not
$p$-equidistributed.

\medskip

{\bf Case 3.} $h$ is a permutation polynomial for $\BF_q$ but not
for an extension $\BF_{q^m}.$  Then we can start with $\BF_{q^m}$
and proceed as in {\bf Case 2.}

Theorem \ref{p-equi} is proved.
\end{proof}

\begin{proof}[Proof of Theorem $\ref{equi}$]
If $f_w$ is  almost noncomposite, then, according to Remarks
\ref{uniform} and \ref{r22}, the word $w$ is equidistributed.

If $f_w$ is  very composite, then for some $p$ it is $p$-composite
but not $p$-special and, by \thmref{p-equi}, the word $w$ is not
$p$-equidistributed. Hence it is not equidistributed.
\end{proof}

\begin{corollary}\label{surj}
Suppose that for each $p$ and all $n$ big enough the image of the map 
$P_{w,p^n}\colon \SL(2,p^n)\times \SL(2,p^n)\to \SL(2,p^n)$ contains
all noncentral semisimple elements of $\SL(2,p^n)$. Then $w$ is
equidistributed.
\end{corollary}

\begin{proof}
Assume the contrary. Then by \thmref{equi}, the polynomial $f_w$ is
very composite, i.e., for some $p$ it is $p$-composite but not
$p$-special. As in Case 2 considered above, we see that for big $n$
the polynomial $f_{w,p^n}$ omits at least $(p^n-1)/d_1$ values. This
contradicts the assumption of the corollary according to which
$f_{w,p^n}$ omits at most two values, 2 and $-2$.
\end{proof}

\begin{remark}
The converse statement is not true. Indeed, let $f_w=\tr(w(x,y))$ be
the trace polynomial of a word $w(x,y)$. Let $a\in\BQ$ be a rational
point in the spectrum $\st(f_w)$, which means that the surface
$H_a$, given by the equation $f_w(s,u,t)=a$, is not absolutely
irreducible. Then for all $p$ big enough the reduction $a_p$ lies in
$\st(f_{w,p})$ (see \cite[2.2.1]{BDN}). It follows that the set of
numbers $q$ such that $H_a(\BF_q)=\emptyset$ may be infinite.

Examples of such words were provided by Jambor, Liebeck and O'Brien 
\cite{JLO}. For instance, let $w(x,y)=x^2(x^2yx^{-2}y^{-1})^2$. Let
us show that
%\begin{Lemma} \label{lem:JLO}
the trace polynomial $f_w$ is $\BQ$-noncomposite.
%\end{Lemma}
%\begin{proof}
Assume to the contrary that $f_w$ is $\BQ$-composite. Then,
according to Proposition \ref{decom} below,
$f_w(s,u,t)=D_2(p(s,u,t))$ for some polynomial $p$, where
$D_2(z)=z^2-2$ is the second Dickson polynomial. We conclude that
$f_w+2=p^2$. However, the factorization of $f_w$ (say, on MAGMA)
shows that $f_w$ is not a full square.
%\end{proof}

On the other hand, $0\in\st(f_w)$ (see, e.g.,
\cite[Lemma~2.2]{JLO}). It is shown in \cite{JLO} that
$H_0(\BF_q)=\emptyset$ for every $q$ such that
\begin{itemize}\item $q=p^{2r+1}$, \, $r\ge 0;$
\item $p\ne 5;$\item $p^2\not\equiv 1  (\bmod 16);$
\item $p^2\not\equiv 1  (\bmod 5).$\end{itemize}

Thus, for these $q$, the morphisms $P_{w,q}$ are  dominant and
equidistributed whereas the elements with zero trace are not in the
range of $P_{w,q}$.
\end{remark}

\begin{remark}
Note that many positive words $w=x^ay^b$, $a>0$, $b>0$, satisfy the
assumptions of Corollary \ref{surj} and are equidistributed, see
\cite{BG} for details (and \cite{LST} for generalizations to simple
groups of higher Lie rank).
\end{remark}

\begin{proof} [Proof of Proposition $\ref{SL-PSL}$]

We may assume that $q$ is odd. Consider the commutative diagram

$$
\xymatrix{  G_q \times  G_q \ar[r]^{P_{w,q}}
\ar[d]^{\rho'} \ar[rd]^{\varkappa} & G_q \ar[d]^{\rho}\\
\hat G_q\times \hat G_q \ar[r]^{\hat P_{w,q}} & \hat G_q}
$$
where $\rho$ and $\rho'$ are natural projections, and $P_{w,q}$ and
$\hat P_{w,q}$ correspond to the map $(x,y)\to w(x,y)$ on $G_q\times
G_q$ and on $\hat G_q\times \hat G_q$, respectively.

Suppose $w$ is $p$-equidistributed with respect to $\{G_q\}$ so that
for $q>q_0$ we have inequalities \eqref{Gq:equi} with parameters
$A$, $B$, $\alpha$, $\beta$. Define $\hat S_q:=\rho( S_q)$, $\hat
T_q:=\hat G_q\setminus\hat S_q.$

For any element $\hat g\in \hat G_q$ the set $\rho^{-1}(\hat g)$
contains precisely two elements  $g_1, g_2$ of $ G_q.$ Therefore,
\begin{itemize}
\item $\#\hat S_q=\#S_q/2=
\#G_q \varepsilon_p(q)/2=\#\hat G_q \varepsilon_p(q);$
\item
$W_{\hat g, \hat G_q}=\rho' (W_{ g_1, G_q}\cup W_{ g_2, G_q});$
\item
$\# W_{\hat g,\hat G_q}=(\#W_{ g_1, G_q}+\#W_{ g_2, G_q})/4;$
\item
for every $\hat g\in \hat T_q$ we have
$$\#W_{\hat g,\hat G_q}=\frac{\#W_{ g_1, G_q}+\#W_{g_2, G_q}}{4}=\#G_q\frac{1+\vareps_p(q)}{2}=\#\hat G_q(1+\vareps_p(q)).$$
\end{itemize}

Hence, $w$ is $p$-equidistributed  on $\{\hat G\}_q$  with the same
parameters as on $\{G_q\}$.
\end{proof}

\begin{remark} \label{measure}
In \cite{GS} there is a discussion on relationship between two close
properties of word maps on finite groups: be equidistributed and
preserve the uniform measure. In our context, the proof of Theorem
\ref{equi} allows us to formulate this relationship explicitly.
\end{remark}

\begin{corollary}\label{measure-sl-psl}
Assume that a word $w$ has an almost noncomposite trace polynomial
$f_w$ of degree  $d.$ Let  $q>4(50d^4)^2$, and let
$\varepsilon(d,q)=3(100d^4+1)q^{-1/2}.$ Let $G=\SL(2,q)$ or
$G=\PSL(2,q).$ Then the word map $w\colon G\times G\to G$ is
$\varepsilon(d,q)$-measure-preserving in the sense of
{\text{\rm{\cite{GS}}}}.
\end{corollary}

\begin{proof}
According to \eqref{parameters}, the word map $w$  is
$\varepsilon(d,q)/3$-equidistributed, in the sense of Definition
\ref{def:equifin}, and hence $\varepsilon(d,q)$-measure-preserving,
by \cite[Proposition~3.2]{GS}.
\end{proof}

Corollary \ref{cor-main} will be proved in Section \ref{sec:comp}.

\section{Composite trace polynomials} \label{sec:comp}

Our goal in this section is to describe words in two variables whose
trace polynomial is composite. A full description could provide an
answer, in the case of $\SL(2)$ and words in two variables, to the
following basic question, which should apparently be attributed to
Larsen and Shalev:

\begin{quest} \label{LS1}
Is it true that a word $w\in F_d$ is equidistributed on a Chevalley
group $G$ (of fixed type) if and only if $w$ is not a proper power
of another word?
\end{quest}

Although our results (summarized in Table 1) are not conclusive,
they give a strong evidence in favour of an affirmative answer to
Question \ref{LS1} in our case. Before explaining the table, we give
some necessary preliminaries.

Throughout this section $D_n(x)$ stands for the $n^{th}$ Dickson
polynomial (see \exampref{ex:st-u}). It is well known (see, e.g.,
\cite[(2.2)]{LMT}) that this polynomial satisfies
$D_n(x+1/x)=x^n+1/x^n$ and is completely determined by this
functional equation.

For the sake of convenience, we define $D_{-n}(x)=D_n(x)$ and
$D_0(x)\equiv 2$. We repeatedly use the decomposition
$D_{nm}(x)=D_n(D_m(x)).$

\begin{notation} \label{not}
We always assume that $w(x,y)$ is written in the form
\begin{equation} \label{canon}
w=x^{a_1}y^{b_1}\dots x^{a_r}y^{b_r}
\end{equation}
and is reduced (all integers $a_i$, $b_j$ are nonzero). We call the
integer $r$ the {\it complexity} of $w$.

If $\BF$ is a field and $f_w\in\BZ[s,u,t]$ is the trace polynomial
of $w$, we keep the same notation for the polynomial $f_w\in
\BF[s,u,t]$ obtained after changing scalars to $\BZ\otimes \BF$.

We denote
\begin{itemize} \item
$A=A(w):=\sum\limits_{i=1}^{r} a_i$, $B=B(w):=\sum\limits_{i=1}^{r}
b_i;$
\item $\bar A=\bar A(w):=\sum\limits_{i=1}^{r} |a_i|$, $\bar B=\bar B(w):=\sum\limits_{i=1}^{r} |b_i|;$
\item  for a polynomial $p(x_1,\dots,x_n)$ we denote by $\deg_{x_i} p$ the degree
of $p$ with respect to the variable $x_i.$
\end{itemize}
\end{notation}

\begin{definition} \label{tr-sim}
Let $w=x^{a_1}y^{b_1}\dots x^{a_r}y^{b_r}$ and
$v=x^{c_1}y^{d_1}\dots x^{c_{r'}}y^{d_{r'}}$ be reduced words written in form
$\eqref{canon}$. We say that they are {\em trace-similar}, and denote this by
$w\approx v$, if $r=r'$, the array $\{|a_i|\}$ is a rearrangement of
$\{|c_i|\},$ and the array $\{|b_i|\}$ is a rearrangement of
$\{|d_i|\}$.
\end{definition}

\begin{prop} \label{Horowitz} \cite{Ho}
If reduced words written in form $\eqref{canon}$ have the same trace
polynomial, then they are trace-similar.
\end{prop}

\begin{example}
The words $w=xy$ and $v=xy^{-1}$ are trace-similar but have
different trace polynomials: $\tr(w)=u$, $\tr(v)=st-u.$ Moreover,
the value sets of the trace polynomials of trace-similar words may
differ: let, say, $w=(xy)^2$ and $v=[x,y]$; the words $w$ and $v$
are trace-similar but $P_v$ is surjective on $SL(2,q)$ whereas $P_w$
is not if $q$ is odd.

The words $x^2y^{-1}xy$ and $x^2yxy^{-1}$ are trace-similar, have
the same trace polynomial but are not conjugate in $F_2$ \cite{Ho}.
\end{example}

We can now explain Table 1. It gives conditions under which one can
conclude that if the trace polynomial $f_w$ is composite then $w$ is
a proper power of another word (or is trace-similar to such a
power). These conditions depend on relations between the degree $n$
of the polynomial $h$ appearing in the decomposition of $f_w$, the
complexity $r$ of the word $w$ (these relations are put in the first
column of the table), and the characteristic $p$ of the ground field
(which is put in the first row). The entries of the table contain
conclusions on $w$ and references to the corresponding assertions.

\begin{table}[h]
\begin{tabular} {|c|c|c|c|}\hline
$f_w=h\circ q$ &             &        $\mathbb F={\mathbb F}_p$,           &    $\mathbb F={\mathbb F}_p$,     \\
$\deg h=n,  $     & $\mathbb F=\BQ$  &   &   \\
compl.$=r$    &   & $p>r$ & $p>r/2$, $p\ne r$ \\ \hline
 $n<r$              & $w\approx v(x,y)^n $ & ?                & ? \\
            & Prop. \ref{sin}&               & \\ \hline
$n=r$ and              & $w=(x^\al y^\be)^r$&  $w=(x^\al y^\be)^r$&  $(w\approx v^r) \Rightarrow (w=(x^\al y^\be)^r$) \\
  ($A\ne 0$ or &Cor. \ref{c=n}       & Prop. \ref{r2}       & Prop. \ref{trsim}         \\
 $B\ne 0$)                 &Prop. \ref{decom}    & Prop. \ref{decom}     &   Prop. \ref{decom} \\ \hline
$n=r$,         & $w=(x^\al y^\be)^r$   &  $w=(x^\al y^\be)^r$ & \\
$r$ prime     &  Cor. \ref{c=n}       &  Cor.  \ref{compos}      &  ? \\
{}            & Prop. \ref{r-prime}     &    {}               &{}\\
{}            & Prop. \ref{decom}       &    {}               &{}\\
\hline
\end{tabular}
\caption{Words with composite trace polynomial} \label{table.sl2}
\end{table}

\begin{prop} \label{prop:BG} \cite{BG}
Let $w$ be a reduced word written in form $\eqref{canon}$, and let
$w_i=x^{a_i}y^{b_i}.$ Then for the trace polynomials
we have $f_{w_i}(s,u,t)=ug_{a_i,b_i}(s,t)+h_{a_i,b_i}(s,t)$, $\deg_s
g_{a_i,b_i}=|a_i|-1$, $\deg_t g_{a_i,b_i}=|b_i|-1.$ Moreover, if
$\BF$ is of characteristic zero or big enough, then
$g_{a_i,b_i}(s,t)\not\equiv 0$ and
\begin{equation}\label{DEC}
f_w(s,u,t)= \sum\limits_{k=0}^r u^k G_k(s,t) \text{ where }
G_r(s,t)=\prod\limits_{i=1}^{r}g_{a_i,b_i}(s,t).
\end{equation}
In particular, $\deg_s G_r=\bar A-r$,  $\deg_t G_r=\bar B-r.$
\end{prop}

\begin{prop} \label{RTGC} \cite{Ri}, \cite{Tu}, \cite{GC}
Let $\BF$ be either $\ov{\BQ}$ or $\ov{\BF}_q$, and let $n$ be a
positive integer. If $p=char(\BF)>0$, assume that $(n,p)=1$. Suppose
that $D_n(x)$ is $\BF$-composite, $D_n(x)=h(g(x))$. Then
$h(x)=D_m(x-c)$ and $g(x)=D_k(x)+c,$ where $km=n$ and $c\in \BF$.
\end{prop}

\begin{remark} \label{rem:RTGC}
The statement of Proposition \ref{RTGC} remains valid if $p$ divides
$n$. Indeed, suppose that $n=kp^s$, $(k,p)=1$, $s\ge 1$. Write $h$
in the form $h(y)=(h_1(y))^{p^t}, $ where $ h_1'\not \equiv 0$ ($t$
may be zero). Denote $r=\deg h$, $r_1=\deg h_1$, then $r=r_1p^t \mid
kp^s$, hence $t\leq s.$

Since $D_n(x)=D_{k}(x)^{p^s}=(h_1(g(x)))^{p^t}$, we have
$D_{k}(x)^{p^{s-t}}=D_{k}(x^{p^{s-t}})=\varepsilon h_1(g(x)),$ where
\ $\varepsilon^{p^t}=1.$ If $s>t$, then the derivative of the
left-hand side is identically zero, and since $h_1'\not \equiv 0,$
we have $g(x)=g_1(x^{p^{s-t}}).$ Let $z=x^{p^{s-t}}.$  Then
$D_{k}(z)=\varepsilon h_1(g(z)). $ Since  $(k,p)=1,$ by \cite{GC} we
have $\varepsilon h_1(z)=D_{r_1}(z-c)$, $g(z)=D_{k/{r_1}}(z)+c.$
Therefore,
$h(z)=(h_1(y))^{p^t}=(D_{r_1}(z-c))^{p^t}=D_{r_1p^t}(z-c).$ If
$s=t$, then $D_k(x)=\varepsilon h_1(g(x))$, and we are under
hypotheses of \propref{RTGC}.
\end{remark}

\begin{remark}
We may and will assume (see \cite{Tu}) that $h(2)=2$ which
corresponds to $c=0.$
\end{remark}

Further on we assume that $\BF$ is either ${\BQ}$ or ${\BF}_p$ (or
the respective algebraic closure, if needed).

Let $q(s,u,t)=uG(s,t)+H(s,t).$  Assume that
\begin{equation}\label{ff1}
q(s,2,s)=2G(s,s)+H(s,s)=g_1(s)+c
\end{equation}
and
\begin{equation}\label{ff2}
q(s,s^2-2,s)=(s^2-2)G(s,s)+H(s,s)=g_2(s)+c.
\end{equation}
Then, if $s\ne\pm 2$, we have
\begin{equation}\label{gg1}
G(s,s)=\frac{g_2(s)-g_1(s)}{s^2-4},
\end{equation}
\begin{equation}\label{gg2}
H(s,s)=\frac{(s^2-2)g_1(s)-2g_2(s)}{s^2-4}+c.
\end{equation}

Indeed, computing $G(s,s)$, $H(s,s)$ from \eqref{ff1} and
\eqref{ff2}, we obtain \eqref{gg1} and  \eqref{gg2}.

In particular, let  $w'(x,y)=x^{a}y^{b}$  and
$f_{w'}(s,u,t)=ug_{a,b}(s,t)+h_{a,b}(s,t).$ Then we have
$$
\begin{aligned}
\tr x^ax^{-b} &= 2g_{a,b}(s,s)+h_{a,b}(s,s)=D_{a-b}(s),\\
\tr x^ax^{b} &= (s^2-2)g_{a,b}(s,s)+h_{a,b}(s,s)=D_{a+b}(s),
\end{aligned}
$$
and, according to \eqref{gg1}, \eqref{gg2}, for $s\ne\pm 2$ we
obtain
$$
\begin{aligned}
g_{a,b}(s,s) &= \frac{D_{a+b}(s)-D_{a-b}(s)}{s^2-4},\\
h_{a,b}(s,s) &= \frac{(s^2-2)D_{a-b}(s)-2D_{a+b}(s)}{s^2-4}.
\end{aligned}
$$
Put  $s=x+x^{-1}$, then
\begin{equation}\label{fab}\begin{aligned}
g_{a,b}(s,s) &=
\frac{(x^{a+b}+x^{-(a+b)})-(x^{a-b}+x^{-(a-b)})}{(x-x^{-1})^2}\\
&= \frac{(x^{a}-x^{-a})(x^{b}-x^{-b})}
{(x-x^{-1})^2}.\end{aligned}\end{equation}

\begin{prop}\label{decom}
With Notation $\ref{not}$, assume that either $A\ne 0$ or $B\ne 0.$
Suppose that $f_w(s,u,t)=h(q(s,u,t))$ where $q\in {\BF}[s,u,t]$ and
$h\in \BF[z]$, $\deg h\ge 2$. Then $h=D_d(z)$ with $d\ge 2$ dividing
both $A$ and $B.$
\end{prop}

\begin{proof}
Putting $y=\id$, $x=\id$, $x=y^{-1}$, and $x=y$, we get,
respectively (taking into account that $\tr (g^{-1})=\tr (g)$):
$$
\begin{aligned}
f_w(s,s,2) & = h(q(s,s,2)) = D_{A}(s), \\
f_w(2,t,t) & =  h(q(2,t,t))  = D_{B}(t), \\
f_w(s,2,s) & =  h(q(s,2,s))  =  D_{A-B}(s), \\
f_w(s,s^2-2,s) & =  h(q(s,s^2-2,s))  =  D_{A+B}(s).
\end{aligned}
$$
These decompositions, together with Proposition \ref{RTGC}, Remark
\ref{rem:RTGC} and the condition $\deg h\ge 2$, imply that there is
a common divisor $d\ge 2$ of all the nonzero numbers from the list
$A$, $B$, $A-B$, $A+B$ such that $h(z)=D_d(z).$
\end{proof}

\begin{prop}\label{sin}
With Notation $\ref{not}$, suppose that $f_w(s,u,t)$ is
$\BF$-composite, $f_w(s,u,t)=h(q(s,u,t))$, where $q\in \BF[s,u,t]$
and  $h(x)=\mu x^n+ \dots$ is a polynomial in one variable of degree
$n$, $\mu\ne 0.$ Then $r=nm$. Moreover, if the characteristic of
$\BF$ is $0$ or big enough, $w(x,y)$ is trace-similar to a word
$v(x,y)^n$ where the complexity of $v$ is $m.$
\end{prop}

\begin{proof}
Let $q(s,u,t)= \sum\limits_{k=0}^m u^k H_k(s,t).$ Then
$$
f_w(s,u,t)=h(q(s,u,t))=\mu u^{mn}H_m^n(s,t)+ \Phi(s,u,t)
$$
where $\deg_u \Phi(s,u,t)< mn.$ Hence $r=nm$ and
$$
\mu
H_m^n(s,t)=G_r(s,t)=\prod\limits_{i=1}^{r}g_{a_i,b_i}(s,t)
$$
(we use the notation of \propref{prop:BG}, in particular, formula
\eqref{DEC}).

Therefore, by \eqref{fab}, we have
$$
\mu H_m^n(s,s)=\pm
\frac{\prod\limits_{i=1}^r(x^{|a_i|}-x^{-|a_i|})(x^{|b_i|}-x^{-|b_i|})}
{(x-x^{-1})^{2r}}.
$$
Let $p=$char$(F)$. If $p>0$, write $|a_i|=\tilde a_ip^{\alpha_i}$,
$|b_j|=\tilde b_jp^{\beta_j}$, with $(\tilde a_i,p)=(\tilde
b_j,p)=1.$ If $p=0$, set $|a_i|=\tilde a_i$, $|b_j|=\tilde b_j$.

Choose an integer $N>\max\{|\tilde b_i|\}$ such that $(N,p)=1$, and
consider the word  $w_N=w(x^N,y).$ Then
$$
f_{w_N}=f_w(D_N(s), \delta(s,u,t),t)=h(q(D_N(s), \delta(s,u,t),t))
$$
where $\delta(s,u,t)=\tr(x^Ny)=ug_{N,1}(s,t)+h_{N,1}(s,t).$ Thus
$$
q(D_N(s), \delta(s,u,t),t)=:q_1(s,u,t)= \sum\limits_{k=0}^m u^k
F_k(s,t)
$$
and
$$
f_{w_N}(s,u,t)=h(q_1(s,u,t))=\mu u^{mn}F_m^n(s,t)+ \Phi_1(u,s,t)
$$
where $\deg_u \Phi_1(u,s,t)< mn.$ Hence, since the words $w$ and
$w_N$ have the same complexity $r$, we have
\begin{equation}\label{p2}
\mu F_m^n(s,s)=\pm \frac{\prod\limits_{i=1}^{r}
(x^{N|a_i|}-x^{-N|a_i|})(x^{|b_i|}-x^{-|b_i|})} {(x-x^{-1})^{2r}}.
\end{equation}

Fix an integer $i\in\{1,\dots ,r\}$. Let $x_0\ne 1$ denote a simple
root of the equation $x^{N\tilde a_i}-1=0$. If $p$ is odd, the order
of zero  $o(x_0)$ of the product in the right-hand side of
\eqref{p2} is equal to the number $\sum_{k\ge 0} n_i(k)p^{k}$ where
$n_i(k)$ denotes the number of appearances of $|\tilde a_i|p^k$ in
the list $|a_1|,\dots ,|a_r|$. On the other hand,
$o(x_0)=n\varkappa_i$ where $\varkappa_i$ is the order of the zero
of $F_m$ at the point  $x_0+1/x_0.$ The same is true for $p=0$ (if
we set $0^0=1$).

Assume that $p > \max\{|a_i|,|b_j|, 1\le i,j\le r\}$, or $p=0$. Then
$n_i(k)= 0 $ for $k>0.$ This means that there are $n\varkappa_i$
appearances of each $|a_i|$ in the list $|a_1|,\dots ,|a_r|.$

In a similar way, looking at the word $w(x,y^M)$ for $M$ big enough
and prime to $p$, we conclude that there are precisely $n\tau_i$
appearances of each $|b_i|$ in the list $|b_1|,\dots ,|b_r|$ where
$\tau_i$ is the order of the corresponding root of $F_m.$ Moreover,
$\sum \varkappa_i=\sum  \tau_i= m.$

Define a word
$$v(x,y)=x^{t_1}y^{k_1}\dots x^{t_m}y^{k_m}$$ of complexity $m$
in such a way that among the $|t_i|$ there will be $\varkappa_i$ of
the $|a_i|$ and among the $|k_i|$ there will be $\tau_i$ of the
$|b_i|.$ By construction, $v^n(x,y)$ is trace-similar to $w(x,y)$
which completes the proof.
\end{proof}

\begin{remark}
In contrast with \propref{decom}, in \propref{sin} we do not exclude
the case $A=B=0$.
\end{remark}

In some particular cases, \propref{sin} provides even more information.

\begin{prop}\label{prop:r2}
With the notation and assumptions of Proposition $\ref{sin}$, assume
also that
\begin{itemize}
\item  $n=r;$
\item $A\ne 0 $ or $B\ne 0;$
\item char $\BF=p>0;$
\item  $p>{r}.$
\end{itemize}
Then $w(x,y)=(x^{\al}  y^{\be})^r$ where $\al=A/r$, $\be=B/r$.
\end{prop}

\begin{proof}
First note that under the hypotheses of the proposition, the
assumptions of \propref{decom} are also satisfied. In particular,
both $A$ and $B$ are divisible by $r$ and hence $\al$ and $\be$ are
integers. We also have

 $$f_w=D_r(q(s,u,t)),  \  q(s,u,t)=G(s,t)u+H(s,t),$$
 $$q(s,2,s)=D_{\al-\be} (s), \
 q(s,s^2-2,s)=D_{\al+\be} (s).$$

Hence, similarly to \eqref{fab}, we have:
$$G(s,s) =\pm
 \frac{(x^{|\al|}-x^{-|\al|})(x^{|\be|}-x^{-|\be|})}
{(x-x^{-1})^2}.$$

It follows that for any $N,M$ we have {\footnotesize{
$$\left(
 \frac{(x^{N|\al|}-x^{-N|\al|})(x^{M|\be|}-x^{-M|\be|})}
{(x-x^{-1})^2}\right)^r=\pm \frac{\prod\limits_{i=1}^{r}
(x^{N|a_i|}-x^{-N|a_i|})(x^{M|b_i|}-x^{-M|b_i|})}
{(x-x^{-1})^{2r}}.$$ }} Hence,
$$
 (x^{|\al|}-x^{-|\al|})^r=\pm\prod\limits_{i=1}^{r}
(x^{|a_i|}-x^{-|a_i|}),$$

$$(x^{|\be|}-x^{-|\be|})^r=\pm\prod\limits_{i=1}^{r}(x^{|b_i|}-x^{-|b_i|}).$$
Comparing the degrees of the corresponding polynomials, we get
$$
\begin{aligned}
|A| & =|\al|r=\sum\limits_{i=1}^{r}|a_i|=\bar{A}, \\
|B| & =|\be|r=\sum\limits_{j=1}^{r}|b_j|=\bar{B}.
\end{aligned}
$$
Hence, all the $a_i$ are of the same sign, and so are all the $b_j$.
Let $|\al|=\tilde\al p^\tau,$ \ $|\be|=\tilde\be p^\varkappa.$
Comparing simple roots of the polynomials, we get
$$|a_i|=\tilde\al p^{k_i}, \  |b_j|=\tilde\be p^{s_i}$$
for every $1\le i\le r$ and every $1\le j\le r.$ Moreover,
\begin{equation}\label {nnn1}
\bar{A}=\tilde\al p^\tau r= \tilde\al\sum_{k\ge 0}
n(k)p^{k},\end{equation} where $n(k)$ denotes the number of
appearances of $|\tilde \al|p^k$ in the list $|a_1|,\dots ,|a_r|$,
and
$$
\bar{B}=\tilde\be p^\varkappa r= \tilde\be\sum_{k\ge 0} m(k)p^{k}
$$
where $m(k)$ denotes the number of appearances of $|\tilde \be|p^k$
in the list $|b_1|,\dots ,|b_r|.$

Consider formula \eqref {nnn1}. Let $K=\max\{k    \ | \ n(k)\ne
0\}.$ Suppose that $\tau>0$. Then
$$
p^\tau r= \sum_{k= 0}^{K} n(k)p^{k}\le\left(\sum_{k= 0}^{K}
n(k)\right)p^K=rp^K.
$$
Thus $\tau\le K.$ It follows that $p^\tau \mid\sum_{k= 0}^{\tau-1}
n(k)p^{k}$, and hence the latter sum equals $sp^\tau$ for some
integer $s.$ On the other hand,
$$
sp^\tau=\sum_{k= 0}^{\tau-1} n(k)p^{k}\le \left(\sum_{k= 0}^{\tau-1}
n(k)\right)p^{\tau-1}\le rp^{\tau-1}<p^\tau.
$$
Contradiction shows that $s=0,$ and $ n(k)=0$  for $k<\tau.$
Dividing \eqref{nnn1} by $  p^\tau$, we get
$$
\sum_{k= \tau}^{K} n(k)p^{k-\tau}=r.
$$
This equality remains true in the case $\tau=0.$ On the other hand,
by the definition of $n(k),$ we have
$$
\sum_{k= \tau}^{K} n(k)=\sum_{k= 0}^{K} n(k)=r.
$$
Hence $n(k)=0$ for $  k\ne \tau,$ i.e., $K=\tau$, $n(\tau)=r$,
$|a_i|=\tilde \al p^\tau=\al.$ Since the $a_i$ are of the same sign,
they are all equal. In a similar way, we conclude that all
$b_j=\tilde \be p^\varkappa=\be$ are equal. Hence $w(x,y)=(x^\al
y^\be)^r.$\end{proof}

\begin{prop}\label{trsim}
Let $w(x,y)=x^ay^b\dots $ be a reduced word of complexity $r$ such
that $f_w(s,u,t)=D_r(q(s,u,t))$, $q\in \BF[s,u,t]$, over $\BF=\BQ$
or some $\BF=\BF_p$ with $p>r/2$, $p\ne r$. If $w(x,y)$ is
trace-similar to $(x^{a}y^{b})^r$, then $w(x,y)=(x^ay^b)^r.$
\end{prop}

\begin{proof}
By assumption, $w$ is the product of syllables  $x^{\pm a}y^{\pm
b}.$

Assume that by cyclic permutation and exchanging roles of $x$ and
$y$ one can modify $w$ to a word $v=v_1\dots v_r$, $v_k=x^{\pm
a}y^{\pm b}$, $k=1,\dots, r$, which contains repeated syllables,
i.e., such that for some $i<j$ we have $v_i=v_j$. Then we consider
the word
$$
\tilde v=v_i\dots v_j\dots v_rv_1\dots v_{i-1}.
$$
The word $\tilde v$ will be called a {\it convenient} form of $w.$
Note that either $f_w(s,u,t)=f_{\tilde v}(s,u,t)$  or
$f_w(s,u,t)=f_{\tilde v}(t,u,s).$ If this procedure is impossible,
we say that $w$ is already in a convenient form. First consider the
case where $a=b=1$  and $\BF=\BQ.$

\begin{Lemma}\label{tr}
Let $w(x,y)=xy\dots x^{\pm 1}y^{\pm 1}=w_1\dots w_r$ be a word in a
convenient form, where $w_i$ are syllables of the form $x^{\pm
1}y^{\pm 1}.$ Let $\BF=\BQ.$

Let $u=\tr(xy), s=\tr(x), t=\tr(y).$ Then
$$
f_w(s,u,t)= \ep u^r-\ep m st u^{r-1}+\dots +g(s,t)
$$
is a polynomial of degree $r$ with respect to $u$ such that
\begin{itemize}
\item
the coefficient at  $u^r$ is $\ep=\pm 1;$
\item $m$ is a non-negative integer, $m\leq r/2$,
and $m=0$ if and only if $w=(xy)^r;$
\item the coefficient at $u^{r-1}$ is $\ep m  st;$
\item the coefficient $f_w(s,0,t)$ at $u^{0}$ is a polynomial $g$ in $s,t$
of total degree strictly less than $2r.$
\end{itemize}
\end{Lemma}

It is important here that we defined $u$ as the trace of the first
syllable.

\begin{proof}[Proof of Lemma $\ref{tr}$]
First consider the case when there are no repeated syllables.

\begin{description}
\item[r=1]
$\tr(xy)=u$.
\item[r=2]\begin{itemize}
\item $\tr(xyxy^{-1})=-u^2+ust-t^2+2, $
\item $\tr(xyx^{-1}y^{-1})=u^2-ust+t^2+s^2-2, $
\item $\tr(xyx^{-1}y)=\tr(yxyx^{-1})=-u^2+ust-s^2+2$.
\end{itemize}
\item[r=3]\begin{itemize}
\item
the words $a_1=x{\bf yx^{-1}yx^{-1}}y^{-1},$ $a_2={\bf xyx}
y^{-1}x^{-1}{\bf y}$ and  \newline $a_3={\bf x}yx^{-1} {\bf
y^{-1}xy^{-1}}$ are not in a convenient form;
\item for $a_4=xyx y^{-1}x^{-1}y^{-1},$ we have
$$
\begin{aligned}
\tr(a_4)& =f_{a_4}(s,u,t)=(ust-u^2-t^2+2)u-\tr(x^3y)\\
& =(ust-u^2-t^2+2)u-u(s^2-2)+(st-u) \\
& = -u^3+stu^2+u(3-t^2-s^2)+st;
\end{aligned}
$$
\item
the word $a_5=xyx^{-1}y^{-1}x^{-1}y$ may be modified to $a_4$ by
cyclic permutation and exchanging roles of $x$ and $y,$ thus
$\tr(a_5)=\tr(a_4);$
\item
$a_6=xyx^{-1}yx y^{-1}$ may be  modified to $a_4$ by cyclic
permutation and changing roles of $y$ and $y^{-1},$ thus
$$\tr(a_6)=f_{a_4}(s,st-u,t)=u^3-u^2st+u(s^2t^2+t^2+s^2-3)+st(4-t^2-s^2).$$
\end{itemize}
\item[r=4]\begin{itemize}
\item
$b_1={\bf xyx} y^{-1}x^{-1}y^{-1}x^{-1}{\bf y}$ and $b_2=x{\bf
yx^{-1}yx^{-1}}y^{-1}x y^{-1}$ are not in a convenient form;
\item
for $b_3=xyx y^{-1}x^{-1}yx^{-1}y^{-1}$  we have
$$
\begin{aligned}
\tr(b_3) & =(ust-u^2-t^2+2)^2-\tr(x^2yx^2y^{-1})
=(ust-u^2-t^2+2)^2 \\
& -[(us-t)(s^2-2)t-(us-t)^2-t^2+2]\\
&=u^4-2u^3st +u^2h_1+uh_2+ (t^2-2)^2+t^2(s^2-2)+2t^2-2,
\end{aligned}
$$
where $h_1,h_2$ are polynomials in $s,t;$
\item
$b_4=xy x^{-1}y  x y^{-1}x^{-1}y^{-1} $ may be modified to $b_3$ by
cyclic permutation, and substituting $x$ by $y^{-1}$ and $y$ by
$x^{-1};$
\item
$b_5=xy x^{-1}y^{-1} x y^{-1}x^{-1}y$ may be modified to $b_3$ by
cyclic permutation, and substituting $x$ by $y$ and $y$ by $x;$
\item
$b_6=xy x^{-1}y^{-1} x^{-1}yx y^{-1}$ may be modified to $b_3$ by
cyclic permutation, and substituting $x$ by $x^{-1}$ and $y$ by $
y^{-1}.$
\end{itemize}
Note that these substitutions do not change $u$, and the coefficient
$m$ is not zero in convenient words.
\end{description}

Any word of complexity $\geq 5$ must have repeated syllables. The
case with repeated syllables will be proved by induction on the
complexity $r.$ Assume that for all words in a convenient form of
complexity $k<r$ the statement of the lemma is valid.

Consider $w(x,y)=w_1\dots w_r$ where $w_1=xy$, $w_i=x^{\pm 1}y^{\pm
1}$, $i=2,\dots ,r$, $w_{j+1}=w_1$, $0<j\leq r-1.$ Thus $w=v_1v_2$
where $v_1=w_1\dots w_{j}$, $v_2=w_{j+1}\dots w_r.$ Denote
$v_3=v_1v_2^{-1},$ it is of complexity $r-2$ since its first
syllable is $xy$ and the last is $(xy)^{-1}.$ By induction
hypothesis,
$$
\begin{aligned}
\tr(v_1) & =\ep_1 u^j-\ep_1 m_1 stu^{j-1}+\dots + g_1, \deg g_1 <
2j;\\
\tr(v_2) & =\ep_2 u^{r-j}-\ep_2 m_2 stu^{r-j-1}+\dots +g_2, \deg
g_2< 2(r-j).
\end{aligned}
$$
The word $v_3$ may not be in a convenient form. This means that
$u=\tr (xy)$ may not be the trace of the first syllable of $v_3.$
Anyway,
$$\tr(v_3)=\ep_3 \hat u^{r-2}-\ep_3 m_3 st\hat u^{r-3}+\dots +g_3, \deg g_3< 2(r-2),$$
where $\hat u $ is either $u$ or $st-u.$  In both cases its degree
with respect to $u$ is at most $r-2$ and the coefficient at $u^0$ is
of total degree at most $2(r-2).$ Therefore
$$
\begin{aligned}
\tr(w) & =\tr(v_1)\tr(v_2)-\tr(v_3) \\
& =\ep_1\ep_2u^r-\ep_1\ep_2st(m_1+m_2)u^{r-1}+\dots +g_1g_2-g_3.
\end{aligned}
$$
Here the total degree of the polynomial $g_1g_2-g_3,$ which is the
coefficient at $u^0,$   is less than $2j+2(r-j)=2r.$ Moreover,
$m_1+m_2$ may be zero only if $m_1=m_2=0,$ which means, by induction
hypothesis, that $v_1=w_1^j$, $v_2=w_1^{r-j}$, so $w=w_1^r.$
\end{proof}

We continue the proof of \propref{trsim}: assume that $\BF=\BQ$  or
$\BF_p$, $p>r.$ Assume that $w(x,y)=w_1 \dots w_r$, where
$w_1=x^ay^b$, $w_i=x^{\pm a}y^{\pm b}$, is written in a convenient
form, and $f_w(s,u,t)=D_r(q(s,u,t)).$ We denote $z=x^a,v=y^b,$ i.e.,
$w(x,y)=\tilde w(z,v),$ and $\tilde w$ is a word of the type
considered in \lemref{tr}. Let $\tilde s=D_{a}(s)$, $\tilde
t=D_{b}(t),$ and $\tilde u=\tr(x^ay^b)= ug_{a,b}(s,t)+h_{a,b}( s,
t),$ where $g_{a,b},h_{a,b}$ are polynomials in $s,t$ and
$g_{a,b}\not\equiv 0$ (see \propref{prop:BG}). Since the polynomial
$q(s,u,t)$ is of degree $1$ with respect to $u,$ we have $q(s,u,t)=
\al(s,t)\tilde u +\be(s,t),$ with rational coefficients $\al$ and
$\be.$ According to \lemref{tr}, we have
\begin{equation}\label{tr2}\begin{aligned}
f_w(s,u,t) & = \ep \tilde u^r-\ep m \tilde s\tilde t \tilde
u^{r-1}+\dots + g(\tilde s,\tilde t)= q^r-rq^{r-2}+\dots
\\
& =(\al(s,t)\tilde u +\be(s,t))^r -r(\al(s,t)\tilde u
+\be(s,t))^{r-2}+\dots
\end{aligned}\end{equation}

Moreover, if $m\ne 0$ then $m\not\equiv 0\pmod p$, since $m\leq
r/2<p. $ It follows that
$$
\al(s,t)=\al={\text{\rm{const}}},  \ \al^r=\ep, \text{ and }
\be(s,t)=-\frac{\ep m \tilde s\tilde t}{r\al^{r-1}}=-\frac{ m\al
\tilde s\tilde t}{r}$$ (division is legitimate because $p\ne r$).
Substituting $q=\al\tilde u-m\al \tilde s\tilde t/r$ into
\eqref{tr2}, we get
$$
\begin{aligned}
f_w(s,u,t) & = \ep \tilde u^r-\ep m \tilde s\tilde t \tilde u^{r-1}+\dots + g(\tilde s,\tilde t)\\
& =\left(\al\tilde u-m \al\tilde s\tilde t/r\right)^r
-r\left(\al\tilde u-m\al \tilde s\tilde t/r\right)^{r-2}+\dots
\end{aligned}
$$

Thus, the coefficient at $(\tilde u)^0$ is a polynomial in $\tilde
s\tilde t$ of total degree $r,$ hence it is a polynomial in $\tilde
s,\tilde t$ of total degree $2r,$ which implies, by  \lemref{tr},
that $\be\equiv 0 $  and $ \tilde w=(zv)^r.$
\end{proof}

\begin{corollary}\label{c=n}
Let $w(x,y)=x^ay^b\dots $ be a reduced word of complexity $r$ such
that $f_w(s,u,t)=D_r(q(s,u,t))$ over $\BQ $. Then
$w(x,y)=(x^ay^b)^r$.
\end{corollary}
\begin{proof}
According to \propref{sin}, every such $w$ is trace-similar to
$(x^ay^b)^r.$ It remains to apply \propref{trsim}.
\end{proof}

\begin{prop}\label{r-prime}
With Notation $\ref{not}$, if $r$ is prime and $w$ is not
$p$-equi\-dis\-tri\-bu\-ted, then $r\ne p$ and at least one of $A$
and $B$ is nonzero.
\end{prop}

\begin{proof}  We maintain the notation of \propref{sin}.
Suppose that $w$ is not $p$-equidistributed. Then its trace
polynomial is ${\BF}_p$-composite, $f_w=h(q(s,u,t))$, and $h$ is not
$p$-special. Since $\deg_u f_w=r$, $h$ is not linear, and $\deg h$
divides $r$, we have $\deg h =r.$ In the notation of \propref{sin},
this means that $n=r$ and $q(s,u,t)=uG(s,t)+H(s,t).$ Consider two
cases.

{\bf Case 1.} $A=B=0.$  Then
$$f_w(s,2,s)=h(q(s,2,s))=D_{A-B}(s)\equiv 2,$$
$$f_w(s,s^2-2,s)=h(q(s,s^2-2,s))=D_{A+B}(s)\equiv 2.$$

Thus, by \eqref{ff1}, \eqref{ff2}, we have $q(s,2,s)\equiv c_1\in
\BF$, $q(s,s^2-2,s)\equiv c_2\in \BF.$ Since $G(s,s)$ is a
polynomial, from \eqref{gg1} for $s\ne\pm 2$ it follows that
$c_1=c_2$, $G(s,s)\equiv 0$, $H(s,s)=$ const. This would mean that
for at least one of the syllables we have $g_{a_i,b_i}(s,s) \equiv
0$, which is impossible for big powers of $p$ (see
\cite[Lemma~2.3]{BG}). It follows that this case does not occur.

{\bf Case 2.} At least one of $A$ and $B$ is not 0 and $r=p.$ In
this case, by \propref{decom} we have $h(z)=D_r(z)$ is a permutation
polynomial, thus  $h$ is $p$-special, contrary to the assumption on
$h$.
\end{proof}

\begin{corollary}\label{compos}
Let $w(x,y)=x^ay^b\dots $ be a reduced word of  prime complexity
$r.$ If $p >r$  and $w$ is not
$p$-equidistributed, then $w=v(x,y)^r.$
 \end{corollary}

\begin{proof} If $w(x,y)$ is not $p$-equidistributed,
then  according to \thmref{p-equi}, \propref{decom}, and
\propref{r-prime}, we have $f_w=D_r(q(s,u,t)).$  By
\propref{r-prime}, either $A\ne 0$ or $B\ne 0.$ By
\propref{prop:r2}, we have $w=(x^\al y^\be)^r$ where $\al=A/r$,
$\be=B/r.$
\end{proof}

\begin{corollary}\label{r2}
The word  $ w(x,y)=x^ay^bx^cy^d$ is either equidistributed or equal
to $(x^ay^b)^2.$
\end{corollary}

\begin{proof}
Suppose that $w$ is not  equidistributed. Then for some prime $p$
its trace polynomial $f_w$ is ${\BF}_p$-composite,
$f_w=h(q(s,u,t))$, and $h$ is not $p$-special. By \propref{r-prime},
$w\ne x^ay^bx^{-a}y^{-b}$ and $p>2.$ Then, by \propref{prop:r2},
$a=c$, $b=d,$ and  $w(x,y)=(x^ay^b)^2.$
 \end{proof}

\section{Generic words} \label{sec:gen}

In this section, we address the following question: picking up a
``generic'' word $w$, should we expect that it is equidistributed?
There is a large body of literature dedicated to the notion of genericity,
and there are several different approaches to this notion. We mostly follow
the setting adopted in \cite{KS}.

\begin{definition} \label{def:gen} (cf. \cite{KS})
Denote by $\cR$ some set of reduced words $w\in F_2$ written in form
$\eqref{canon}$. For a word of complexity $r$, let $\ell
(w)=\sum_{i=1}^r(|a_i|+|b_i|)$ denote the length of $w$. Let
$S\subseteq \cR$. Set
$$
\rho (n,S)=\#\{w\in S : \, \ell(w)\le n\},
$$
$$
\mu (n,S)=\frac{\rho (n,S)}{\rho(n,\cR)}.
$$

We say that $S$ is
\begin{itemize}
\item
{\em generic} if $ \lim_{n\to\infty} \mu (n,S)=1, $
\item
{\em exponentially generic} if it is generic and the convergence is exponentially fast,
\item
{\em negligible} if this limit equals $0$,
\item
{\em exponentially negligible} if it is negligible and the convergence is exponentially fast.
\end{itemize}
\end{definition}

Evidently, $S$ is (exponentially) generic if and only if the
complement $\cR\setminus S$ is (exponentially) negligible.

\begin{prop} \label{prop:gen}
Let $\cR$ be the set of words $w$ of {\em prime} complexity. Then
the set $S$ of words $w\in\cR$, such that the corresponding morphism
$\BP_w\colon \SL_{2,\BZ}\times \SL_{2,\BZ}\to \SL_{2,\BZ}$ is
$p$-equi\-dis\-tri\-bu\-ted for all but finitely many primes $p$, is
exponentially generic in $\cR$.
\end{prop}

\begin{proof}
Let $w\in \cR$. Suppose that $w\notin S$, i.e., there exist
infinitely many primes $p$ such that the word morphism $P_w$ is {\it
not} $p$-equidistributed. Denote by $\cP$ the set of all such
primes. By \corolref{compos}, $w=(x^ay^b)^r.$

 It remains to refer to \cite{AO} where
it is proven that the property of a word to be a proper power of
another word is exponentially negligible. Hence $S$ is exponentially
generic in $\cR$.\end{proof}

\begin{remark} \label{rem:hope}
We believe that with some more effort, one can significantly
streng\-then \propref{prop:gen}, in particular, by dropping the
primality restriction on the complexity. We leave this to experts in
word combinatorics.
\end{remark}

\section{Concluding remarks} \label{concl}

It is tempting to generalize our results in the following directions:
\begin{itemize}
\item[(i)]
extend them from words in two letters to words in $d$ letters, $d>2$;
\item[(ii)]
keep $d=2$ but consider arbitrary finite Chevalley groups;
\item[(iii)]
combine (i) and (ii).
\end{itemize}

Whereas in case (i) one can still hope to use trace polynomials,
which exist for any $d$, to produce criteria for equidistribution,
cases (ii) and (iii) require some new terms for formulating such
criteria and new tools for proving them.

Regardless of getting such criteria, it would be interesting to
compare, in the general case, the properties of having large image
and being equidistributed, in the spirit of Corollary \ref{surj}. We
dare to formulate the following conjecture.

\begin{conj} \label{conj:large}
For a fixed $p$, let $G_q$ be a family of Chevalley groups of fixed
Lie type over $\BF_q$ $(q=p^n$ varies$)$. For a fixed word $w\in
F_d$, $d\ge 2$, let $P_q=P_{w,q}\colon (G_q)^d\to G_q$ be the
corresponding map. Suppose that

\noindent $(*)$ for all $n$ big enough the image of $P_q$ contains
all regular semisimple elements of $G_q$.

Then the family $\{P_q\}$ is almost $p$-equidistributed.
\end{conj}

It is a challenging task to describe the words $w$ satisfying
condition (*) in Conjecture \ref{conj:large} (cf. the discussion in
\cite{LST} after Theorem~5.3.2). Certainly, words of the form
$w=v^k$, $k\ge 2$, do not satisfy this condition. We do not know any
non-power word for which (*) does not hold.

\medskip

One can try yet another direction: consider equidistribution
problems for matrix {\it algebras} and for polynomials more general
than word polynomials (see Introduction). Even the case of $2\times
2$-matrices is completely open.

\medskip

{\it Acknowledgements.} The authors were supported in part by the
Minerva Foundation through the Emmy Noether Research Institute for
Mathematics.  Kunyavski\u\i \ was supported in part by grant 1207/12
of the Israel Science Foundation. A part of the work was done during
the visit of the second author to the MPIM (Bonn). Support of these
institutions is gratefully appreciated.

We thank S.~Garion, I.~Kapovich, M.~Larsen, A.~Shalev, and
Yu.~Zarhin for helpful discussions and G.~L.~Mullen for providing
reference \cite{GC}.

\end{document}